\documentclass[reqno]{amsart}
\usepackage{amssymb,amscd,verbatim, amsthm,graphics, color,latexsym,amsmath,multicol}
\usepackage[all,cmtip]{xy}
\usepackage{fancybox}

\usepackage[top=3cm, bottom=3cm, left=3.2cm, right=3.2cm]{geometry}

\newtheorem{thm}[equation]{Theorem}

\newtheorem{lemma}[equation]{Lemma}
\newtheorem{Expectation}[equation]{Expectation}

\newtheorem{prop}[equation]{Proposition}
\newtheorem{cor}[equation]{Corollary}
\newtheorem{hyp}[equation]{Hypothesis}

\newtheorem{defn}[equation]{Definition}
\newtheorem{conj}[equation]{Conjecture}

\newtheorem*{thm*}{Theorem}
\newtheorem*{rmk*}{Remark}
\theoremstyle{remark}
\newtheorem{remark}[equation]{Remark}
\theoremstyle{remark}

\newtheorem{sketch}[equation]{Sketch of the proof}

\newcommand \C[1]{{\mathcal #1}}

\newcommand \bC{{\mathbb C}}
\newcommand \bF{{\mathbb F}}

\newcommand \bR{{\mathbb R}}
\newcommand \RR{{\mathbb R}}
\newcommand \bZ{{\mathbb Z}}
\newcommand \ZZ{{\mathbb Z}}

\newcommand \QQ{{\mathbb Q}}

\newcommand \Qlb{\bar{{\mathbb Q}}_\ell}
\newcommand \ra{{\rightarrow}}
\newcommand{\isoarrow}{{~\overset\sim\longrightarrow~}}

\newcommand\CA{{\C A}}
\newcommand\CB{{\C B}}
\newcommand\CH{{\C H}}
\newcommand\CC{{\C C}}
\newcommand\CL{{\C L}}

\newcommand\CU{{\C U}}

\newcommand\ep{{\epsilon}}

\newcommand\om{{\omega}}
\newcommand\al{{\alpha}}

\newcommand\ft{{\mathfrak t}}

\newcommand\Hom{\operatorname{Hom}}

\newcommand\Ind{\operatorname{Ind}}

\newcommand\ad{\mathbf{A}}
\newcommand\hG{G^\vee}
\newcommand\LGr{{}^LG}
\newcommand\LGad{{}^LG^{\rm ad}}

\newcommand\val{\mathsf{val}}

\def\<{\langle} 
\def\>{\rangle}

\numberwithin{equation}{section}

\begin{document}

\author{Dan Ciubotaru}
\thanks{D. Ciubotaru was partially supported by EPSRC EP/V046713/1(2021).
}

\address{Dan Ciubotaru\\Mathematical Institute, Oxford University, Oxford OX2 6GG, UK
}
 \email{dan.ciubotaru@maths.ox.ac.uk}

\author{Michael Harris}
\thanks{M. Harris  was partially supported by NSF Grants DMS-2001369 and DMS-1952667, and by a Simons Foundation Fellowship, Award number  663678
}

\address{Michael Harris\\
Department of Mathematics, Columbia University, New York, NY  10027, USA}
 \email{harris@math.columbia.edu}
 
 \title[On the generalized Ramanujan conjecture over function fields]{On the generalized Ramanujan conjecture over function fields}

\maketitle

\begin{abstract} Let  $G$ be a simple group
over a global function field $K$, and let $\pi$ be a cuspidal automorphic representation of $G$.  Suppose 
$K$ has two places $u$ and $v$ (satisfying a mild restriction on the residue field cardinality), at which the group $G$ is quasi-split,
such that  $\pi_u$ is  tempered and $\pi_v$ is unramified and generic.   We prove that $\pi$ is tempered at all unramified places $K_w$ at which $G$ is unramified quasi-split.%, unless $G(K_w)$ is the quasi-split form of a unitary group in $2n+1$ variables. 

The proof uses the Galois parametrization of cuspidal representations due to V. Lafforgue to relate the local Satake parameters of $\pi$ to Deligne's theory of Frobenius weights.  The main observation is that, in view of the classification of generic unitary spherical representations, due to Barbasch and the first-named author, the theory of weights excludes generic complementary series as possible local components of $\pi$.  This in turn determines the local Frobenius weights at all unramified places.  In order to apply this observation in practice we need a result of the second-named author with Gan and Sawin on the weights of discrete series representations.
\end{abstract}

\section{Introduction}   One of the most striking implications of Laurent Lafforgue's proof of the Langlands correspondence for $GL(n)$ over function fields is the Ramanujan Conjecture:  If $K$ is a global function field over a finite field and $\pi$ is a cuspidal automorphic representation of $GL(n,K)$ with unitary central character, then the local components of $\pi$ are tempered everywhere.  This is believed to be the case as well when $K$ is a number field, but Lafforgue's proof relies on Deligne's purity theorem (also known as the Weil Conjectures), whose applications over number fields are limited.

Vincent Lafforgue has used the geometry of varieties (or stacks) over finite fields to construct global Langlands parameters for 
cuspidal automorphic representations of $G(K)$ when $G$ is any connected reductive group.  In particular, Deligne's purity theorem is used at least implicitly in much of the theory on which his construction relies.   It is therefore natural to ask whether his results imply anything like the Ramanujan Conjecture for general $G$.  The first observation is that it has long been known that the most direct translation of the Ramanujan Conjecture fails for groups other than $GL(n)$ and its inner forms;  there are numerous constructions of cuspidal automorphic representations of classical groups that are not locally tempered almost everywhere.  The Arthur Conjectures were designed in part to address this phenomenon and to characterize the discrepancy from temperedness systematically.   See \cite{Sa} for a review of the generalized Ramanujan Conjecture, primarily in the context of number fields.

Shahidi's $L$-packet conjecture provides an important insight into the question.  The conjecture, in its recent refinement \cite[Conjecture 1.5]{LS}, asserts that every Arthur
packet of a quasi-split $G$ is tempered if and only if it contains a locally generic member with respect to any Whittaker datum.  
In view of the Arthur Conjectures, in other words, genericity, in some form, is the natural condition for a version of the Ramanujan 
Conjecture valid for general $G$.   In his article \cite{Sh}, Shahidi showed\footnote{The article \cite{LS} of B. Liu and Shahidi treats more general Arthur packets, in connection with a conjecture due to Jiang.}  that 
\begin{thm*}\cite[Theorem 6.2]{Sh}  Assume $G$ is a quasi-split connected reductive group over the global field $K$.  The Arthur Conjectures, together with a generalization due to Clozel \cite{Cl}, imply that if $\pi$ is 
an everywhere locally generic cuspidal automorphic representation of $G$ with unitary central character, then the local components $\pi_v$ are tempered for all but finitely many $v$.
\end{thm*}
Shahidi states his theorem for number fields but it should be valid for function fields as well.\footnote{Laurent Clozel has pointed out that, in his review of Shahidi's article for Mathematical Reviews, Wee Teck Gan noted that the version of the Arthur Conjectures assumed in Shahidi's Theorem 6.2 includes the Ramanujan Conjecture for $GL(n)$.  For number fields this is quite far from being established, but for function fields this is due  to Laurent Lafforgue \cite[Th\'eor\`eme VI.10]{Laf02}.}

The purpose of this paper is to prove a version of this theorem when $K$ is a function field.   Our theorem is unconditional -- we assume neither the Arthur Conjectures nor
Clozel's generalization -- and we assume only that $\pi_v$ is generic at one place, and that $\pi_u$ is tempered at a second place.   This is consistent with the version of
the generalized Ramanujan Conjecture that predicts that {\it globally generic} cuspidal automorphic representations are everywhere tempered -- in \cite{Sa} Sarnak attributes this
version to the article \cite{HPS} of Howe and Piatetski-Shapiro.    At present we are required to make  a few additional technical assumptions.  Here is our statement.

\begin{thm*}[Theorem \ref{mainthm}, below]\label{t:main-intro}  Let $K$ be a function field over a finite field of characteristic $p$, and let
$G$ be a connected  simple group over $K$.  Let $\pi$ be a cuspidal automorphic representation of $G$.  Suppose 
\begin{enumerate}
\item The local component $\pi_v$ is generic and unramified at a place $v$ of $K$ such that $G(K_v)$ is an unramified quasi-split group
\footnote{Here and elsewhere, we use this expression as shorthand
for the condition that $G$, as an algebraic group over $K_v$, is unramified and quasi-split.}, other than the quasi-split form of a unitary group in $2n+1$ variables;
\item There is a place $u$ of $K$ such that $\pi_u$  is tempered;
\item  If $G$ is of type $B_n$ or $C_n$, we assume $p > 2$.  If $G$ is of type $F_4$ or $G_2$, we assume $p > 3$.  
%\item If $G$ is of type $C_n$ or $D_n$, we assume that $G$ is adjoint.
\end{enumerate}
Then for every place $w$ of $K$ at which $\pi$ is unramified and $G(K_w)$ is unramified quasi-split, the local component $\pi_w$ is tempered.
\end{thm*}

\begin{rmk*}  Strictly speaking, $\pi$ is defined over a number field and the appropriate notion, both in the hypothesis and in the conclusion, is $\iota$-tempered for some embedding $\iota:  \bar{\QQ} \hookrightarrow \bC$.
For a discussion of the relevant notions of temperedness see Section \ref{iotatempered} and Theorem \ref{mainthm}.  
\end{rmk*}

The proof comes down to assembling three basic structural features of three distinct aspects of the theory of automorphic representations, together with a recent result of the second-named author with Gan and Sawin that serves to hold the disparate parts of the argument together.  The first element of the assembly is V. Lafforgue's assignment to $\pi$ of a compatible family of semisimple $\ell$-adic Langlands parameters 
$$\sigma = \sigma_{\ell,\pi}:   Gal(K^{sep}/K) \ra \LGr(\Qlb),$$
where $\LGr$ is the Langlands dual group.  The properties of $\sigma$ are recalled in \S \ref{sec_param}; the important point is that the restriction of $\sigma$ to the decomposition group at an unramified place $w$ belongs to the conjugacy class of the Satake parameter of $\pi_w$.

The second element is Deligne's theory of weights of $\ell$-adic representations of the Galois group of a function field over a finite field.  Together with L. Lafforgue's proof of the global Langlands correspondence for $GL(n,K)$, Deligne's theory implies that the composition of $\sigma$ with any $N$-dimensional linear representation  
$$\tau:  \LGr \ra GL(N,\Qlb),$$
a semisimple rank $N$ $\ell$-adic local system, decomposes uniquely as the direct sum of punctually pure local systems of  (integer) weights $w_\chi$, each shifted by a character $\chi$ of the Galois group of the constant field that can be thought of as the non-integral part of the summand.  A more precise statement is given in \eqref{chidecomp} below.   

If we knew that $\tau \circ \sigma$ were irreducible for some $\tau$,  it would be attached by the Langlands correspondence for $GL(N)$ to a cuspidal automorphic representation of $GL(N)$, and Laurent Lafforgue's results, combined with the semisimplicity of $G$, would imply that no non-trivial $\chi$ appears in the direct sum decomposition.  The genericity hypothesis (1) in the statement of Theorem \ref{mainthm} should imply that some $\tau\circ\sigma$ is irreducible, but this is not known in general.  Instead,
hypothesis (2) in the statement of Theorem \ref{mainthm} provides the anchor that guarantees that only integer weights occur in 
$\tau\circ\sigma$.   With this in hand, we can apply the classification of generic spherical unitary representations of the split group $G(K_v)$.   Since every irreducible spherical representation factors through the adjoint group, it is sufficient to consider the classification of the spherical unitary dual when $G(K_v)$ is adjoint. This is due to Barbasch and the first-named author and its application to the problem at hand takes up most of the length of the article.  By analyzing $\tau\circ \sigma$ when $\tau$ is the adjoint representation, and one additional minuscule representation where needed, a case-by-case consideration of the classification shows that no generic complementary series representation is compatible with the requirement that the weights be integral.  This in turn implies that $\tau\circ\sigma$ is punctually pure of weight $0$, which implies Theorem \ref{mainthm}.   

The restriction on $p$ is an artifact of the proof of the theorem in \cite{GHS} and can undoubtedly be eliminated.  
 If $G(K_v)$ is the quasi-split form of a unitary group in $2n+1$ variables, the analysis of the generic spherical unitary locus  can only conclude that if the local component $\pi_v$ in Main Theorem is not  tempered, then the real part of its Satake parameter is precisely one explicitly determined parameter, $\frac 12\omega_n$, see Remark \ref{PSU}.
%The restriction on the isogeny when $G$ is of type $C_n$ or $D_n$, and the exclusion of the quasi-split nonsplit form of the odd unitary groups are necessary for the following reason. When $G(K_w)$ is split of type $C_n$ or $D_n$ (and similarly for the quasi-split nonsplit form of an odd unitary group, see Remark \ref{PSU}), the analysis of the generic spherical unitary locus  can only conclude that if the local component $\pi_w$ in Main Theorem is not  tempered, then the real part of its Satake parameter is precisely one explicitly determined parameter, $\frac 12\omega_n$ in the notation of  Theorem \ref{t:half-adjoint}. This vexing parameter can then only be ruled out by specializing $\tau$ to a spin representation of the Langlands dual $G^\vee$, and this is where we need $G^\vee$ to be simply connected, i.e., a spin group.  

As a corollary of the Main Theorem, we obtain that every unramified local component $\pi_w$ is generic, provided $G(K_w)$ is {unramified quasi-}split   subject to the same restrictions as above.  
Readers are referred to Shahidi's article \cite{Sh} for indications of the compatibility of our results with the Arthur Conjectures.  

\section*{Acknowledgments}  The authors thank Laurent Clozel, Johan de Jong, H\'el\`ene Esnault, Dennis Gaitsgory, Mark Kisin, Colette Moeglin, Peter Sarnak, Freydoon Shahidi, and Jack Thorne for  helpful clarifications.  We also thank Martin Olsson for explaining the current state of the literature on weights in the $\ell$-adic cohomology of stacks, and for providing the sketch of the proof of the purity result used to provide an alternative approach to Theorem \ref{mainthm}.

\section{Parameters of cuspidal representations and Weil numbers}\label{sec_param}

In what follows $k_1$ is a finite field of order $q_1 = p^f$ for some prime $p$, $C$ is an algebraic curve over $k_1$ with function field
$K = k_1(C)$, $v$ is a chosen place of $K$ with residue field $k$ of order $q$, and $G$ is a semisimple group over $K$ that is {unramified quasi-}split at $v$.  {For
reasons to be explained in Remark \ref{PSU}, we assume that $G(K_v)$ is not isomorphic to the quasi-split but non-split group $PSU(2n+1)$. }

We will fix
an open compact subgroup $U \subset G(\ad_K)$ that contains a hyperspecial maximal compact subgroup $U_v \subset G(K_v)$.   Let $X_U$ denote the double coset space $G(K)\backslash G(\ad_K)/U$.
The space of automorphic forms  of level $U$ with coefficients in a (commutative) ring $E$ is denoted
$$\CA(G,U,E) := \{f:  X_U \ra E \},$$
where all $f$ are assumed continuous with respect to the discrete topology on $E$.  With this definition it is clear that if $E$ is a subring of the algebraically
closed field $\mathcal{C}$
then the inclusion of $E$ in $\mathcal{C}$ defines an isomorphism
$$\CA(G,U,E)\otimes \mathcal{C} \isoarrow \CA(G,U,\mathcal{C}).$$
 
Assume $p$ is invertible in $E$ and let $\CA_0(G,U,E) \subset \CA(G,U,E)$ denote the submodule of cusp forms.  This is defined by the vanishing of
constant terms, which are defined by integrals with respect to a measure with values in $\ZZ[\frac{1}{p}]$.   We let 
$$\CA_0(G,E) = \varinjlim_{U} \CA_0(G,U,E),$$
the colimit taken with respect to inclusion.  
 
 \begin{prop}\label{cuspfinite}  Suppose $E$ is a noetherian $\ZZ[\frac{1}{p}]$ algebra with a fixed embedding in the algebraically
 closed field $\mathcal{C}$.
 Then $\CA_0(G,U,E)$ is a finite $E$-module and the functions in $\CA_0(G,U,E)$ are compactly supported in $X_U$.
 Moreover, $\CA_0(G,U,E)\otimes \mathcal{C} \isoarrow \CA_0(G,U,\mathcal{C}).$
  \end{prop}
  
  \begin{proof}  All but the final statement is contained in \cite[Proposition 8.2]{BHKT}; the last statement is obvious.
  \end{proof}
  
  In the applications, $E$ will be a number field and $\mathcal{C}$ will be either $\bC$ or $\Qlb$.  We fix an algebraic closure
  $\bar{\QQ}$ of $\QQ$ containing $E$ and choose embeddings
  $$\iota:  \bar{\QQ} \hookrightarrow \bC; ~~~ \iota_\ell:  \bar{\QQ} \hookrightarrow \Qlb.$$

\subsection{Rationality of cuspidal representations}  Let $S = S(U)$ be the finite set of places $w$ of $K$ such that $U_w := U \cap G(K_w)$ is not a hyperspecial
maximal compact subgroup; thus $v \notin S$.   We define the Hecke algebra 
$$\CH^S(G,U) = \otimes'_{w \notin S} \CH_w$$
where $\CH_w = \CH(G(K_w),U_w)$ is the algebra of $\QQ$-valued compactly supported functions on $G(K_w)$, biinvariant under $U_w$, with multiplication given
by convolution; the restricted tensor product $\otimes'$ is taken with respect to the unit elements $\mathbf{1}_{U_w} \in \CH_w$.   This is a commutative algebra and for any
$E$ as in Proposition \ref{cuspfinite} the action of $\CH^S(G,U)\otimes E$ on $\CA_0(G,U,E)$ is semisimple and decomposes $\CA_0(G,U,E)$ as a finite sum of
eigenspaces.  It follows in particular that 

\begin{lemma}\label{algHecke}  The eigenvalues of $\CH^S(G,U)$ on $\CA_0(G,U,\mathcal{C})$ are algebraic numbers for any algebraically closed $\mathcal{C}$.
\end{lemma}

Similarly, the action of $G(\ad_K)$ on cusp forms preserves the rational subspace $\CA_0(G,\QQ)$.  So any  irreducible cuspidal automorphic representation has a model
over a number field.  It follows that 
$$\CA_0(G,\bar{\QQ}) \isoarrow \bigoplus_{\pi}  m(\pi) \pi;   \CA_0(G,U,\bar{\QQ}) \isoarrow \bigoplus_{\pi} m_U(\pi) \pi^U$$
where $\pi$ runs over the irreducible $\bar{\QQ}$-representations of $G(\ad_K)$ and $m(\pi) \geq 0$ is an integer multiplicity, and we make the convention
that $m_U(\pi) = 0$ if $\pi^U = 0$.  

Fix a prime $\ell \neq p$ and a level subgroup $U$.  Let $\CH^S(G,U),\Qlb) = \CH^S(G,U)\otimes \Qlb$.   In \cite{Laf18}
Vincent Lafforgue defines a commutative algebra $\CB(G,U,\Qlb) \supset \CH^S(G,U,\Qlb)$ of {\it excursion operators} that acts on $\CA_0(G,U,\Qlb)$ and defines a
decomposition
\begin{equation}\label{param}
\CA_0(G,U,\Qlb) \isoarrow \bigoplus_{\sigma} \CA_0(G,U,\Qlb)_{\lambda(\sigma)},
\end{equation}
where $\lambda(\sigma)$ runs over characters of $\CB(G,U,\Qlb)$ that occur non-trivially in $\CA_0(G,U,\Qlb)$, and $\sigma$ designates a semisimple Langlands parameter
\begin{equation}\label{sigma}
\sigma = \sigma_\ell:  Gal(K^{sep}/K) \ra \LGr(\Qlb).
\end{equation}
It is not known in general whether or not this decomposition is defined over $\bar{\QQ}$, but it is known for $GL(n)$ (when an appropriate central character is fixed).  However,
\begin{lemma}\label{excHecke}  The restriction of the action of $\CB(G,U,\Qlb)$ on $\CA_0(G,U,\Qlb)$ to the subalgebra $\CH^S(G,U,\Qlb)$ coincides with the $\Qlb$-linear
extension of the natural action of $\CH^S(G,U,\bar{\QQ})$.
\end{lemma}

Let $\Phi(U)$ denote the set of $\sigma$ such that $\CA_0(G,U,\Qlb)_{\lambda(\sigma)} \neq 0$.    For $\sigma \in \Phi(U)$, we let $\lambda_\CH(\sigma)$ denote
the restriction of $\lambda(\sigma)$ to $\CH^S(G,U,\bar{\QQ})$; for $w \notin S(U)$ we let  $\lambda_w(\sigma)$ denote the restriction of $\lambda(\sigma)$ to $\CH_w$.  
We let $\LGr^{ss}(\Qlb)$ denote the set of semisimple elements of $\LGr(\Qlb)$, $[\LGr^{ss}(\Qlb)]$ the set of semisimple conjugacy classes in $\LGr(\Qlb)$.
For $\sigma \in \Phi(U)$ and $w \notin S(U)$, let 
$$\alpha_w(\sigma) = [\sigma(Frob_w)] \in [\LGr^{ss}(\Qlb)],$$
where the brackets around $\sigma(Frob_w)$ denote the conjugacy class of the image under $\sigma$ of a choice of Frobenius element at $w$.

The key fact about Lafforgue's parametrization that we need is
\begin{thm}[Lafforgue]\label{SeqT}  Fix $\sigma \in \Phi(U)$.  Let
$s_{w,\ell}(\sigma) \in [\LGr^{ss}(\Qlb)]$ be the Satake parameter corresponding to $\lambda_w(\sigma)$.  Then 
$$s_{w,\ell}(\sigma) = \iota_\ell(s_w(\sigma))$$
for an $s_w(\sigma) \in [\LGr^{ss}(\bar{\QQ}]$
and
$$\alpha_w(\sigma) = s_w(\sigma).$$
\end{thm}

On the other hand, $\lambda_v(\sigma)$, with our fixed unramified place $v$, is the action of $\CH_v$ on $\pi^U$ for some $\pi \subset \CA_0(G,\bar{\QQ})$ with $m_U(\pi) \neq 0$.
Since $\pi$ is cuspidal, it is unitary, and therefore $\iota(s_v(\sigma)) \in [\LGr^{ss}(\bC)]$ is the Satake parameter of a unitary spherical 
representation $\pi_v$ of $G(K_v)$.    Let $T^\vee$ be a maximal torus of the Langlands dual group $\hG$ of $G$.  Since $G(K_v)$ is a split group we may take its Satake parameter $s_v(\sigma)$ to be an element of
$[T^{\vee}(\bar{\QQ})]$, the set of algebraic semisimple conjugacy classes in $\hG(\bar{\QQ})$; the brackets denote conjugacy classes under the Weyl group of $\hG$.
It follows from Theorem \ref{SeqT} that we thus have an algebraic semisimple conjugacy class $s_v(\sigma)$ in $T^\vee$ whose image under a complex embedding corresponds to a unitary spherical representation, while its image under an $\ell$-adic embedding corresponds to the conjugacy class of the Frobenius at $v$ of
a semisimple Galois parameter with values in $\hG(\Qlb)$.  The interaction between these properties of the algebraic conjugacy class 
$s_v(\sigma)$ is the subject of our main theorems.

\begin{defn}\label{belongs}  Let $\pi \subset \CA_0(G,\bar{\QQ})$ with $m_U(\pi) \neq 0$.  We say $\pi$ {\bf belongs to} the parameter $\sigma$ if the restriction of
$\lambda_\CH(\sigma)$ coincides with the action of $\CH^S(G,U,\bar{\QQ})$ on $\pi^U$.  
\end{defn}

{\subsection{Tempered and $\iota$-tempered representations}\label{iotatempered}

Let $F = k((t))$; in other words $F = K_v$ in the previous notation, and let $\pi$ be an irreducible admissible representation of $G(F)$ with coefficients in $\bC$.  Then $\pi$ is tempered
if and only if there is a parabolic subgroup $P \subset G(F)$ with Levi component $M$ and a discrete series representation $\rho$ of $M$, with unitary central character,
such that $\pi$ is an irreducible constituent of the (completely reducible) parabolically unitarily induced representation $\Ind_P^{G(F)} \rho$.  

\begin{defn}\label{tempitemp} Let $\pi$ be an irreducible admissible representation of $G(F)$ with coefficients in $\bar{\QQ}$.  Then we define $\pi$ to be {\it $\iota$-tempered} if $\pi \otimes_{\bar{\QQ},\iota} \bC$ is tempered  in the above sense.  We say $\pi$
is {\it tempered} if it is $\iota'$-tempered for every embedding $\iota':   \bar{\QQ} \hookrightarrow \bC$.
\end{defn}

Suppose $\pi$ as above is $\iota$-tempered, and let $Z_M$ be the center of $M$.  Then the central character 
$$\xi_\rho:  Z_M \ra \bar{\QQ}^\times,$$
has the property that $\iota\circ \xi_\rho:  Z_M \ra \bC^\times$ is unitary.  There is no reason for this to remain true if $\iota$ is replaced by a second embedding
$\iota':  \bar{\QQ} \hookrightarrow \bC$.  The following should be a consequence of the Arthur conjectures, but we have not been able to convince ourselves
that we can prove the conjecture:

\begin{conj}\label{iotat}  Let $\Pi$ be a cuspidal automorphic representation of $G(\ad_K)$ with coefficients in $\bar{\QQ}$.   Suppose $\Pi_v$ is $\iota$-tempered.
Then $\Pi_v$ is tempered.
\end{conj}

}

\section{The main theorems}

Let $\ft^\vee$ denote a Cartan subalgebra of $Lie(\hG)$, $\ft^\vee_\RR$ its real form.  Let $\sigma \in \Phi(U)$  be a parameter as in Theorem \ref{SeqT}, and
let $\pi \subset \CA_0(G,\bar{\QQ})$ be an irreducible representation that belongs to $\sigma$, as in Definition \ref{belongs}.  For any place $w \notin S(U)$ we let $s_w(\sigma)$ 
denote the component in $\hG$ of the local Satake parameter of $\pi$ at $w$.  We write 
\begin{equation}\label{logq}  \nu_w(\sigma) = \log_q(|\iota(s_w(\sigma)|) \in \ft^\vee_\RR.
\end{equation}

We record the following well-known fact, see for example \cite[\S7]{Re}.
\begin{lemma}\label{temperedprin}  Let $\sigma$ be as above.  Then the representation $\pi_w$ is tempered
if and only if $\nu_w(\sigma) = 0$.  Moreover, $\pi_w$ is then an irreducible summand of a principal series representation unitarily induced from a unitary character of the minimal
parabolic subgroup of $G(K_w)$.
\end{lemma}

For any irreducible $M$-dimensional representation $\tau$ of $\LGr$, the composition 
$$\tau\circ\sigma:  Gal(K^{sep}/K)~ \ra ~GL(M,\Qlb)$$
corresponds to a semisimple $\ell$-adic local system, denoted $\CL(\sigma,\tau)$, on the curve $C \setminus |S|$ over the finite field $k_1$.   We write
$\CL(\pi,\tau) = \CL(\sigma,\tau)$ if $\pi$ belongs to $\sigma$.    We make the following hypotheses on the characteristic:

\begin{hyp}\label{Kloo}  If $G$ is of type $B_n$ or $C_n$,
we assume $p > 2$.  If $G$ is of type $F_4$ or $G_2$, we assume $p > 3$. % If $G$ is of type $C_n$ or $D_n$, we assume that $G$ is adjoint.
\end{hyp}

We now state our main result.

\begin{thm}\label{mainthm}  Let $\pi \subset \CA_0(G,\bar{\QQ})$ with $m_U(\pi) \neq 0$, and suppose
$\pi$ belongs to the parameter $\sigma$.  We assume Hypothesis \ref{Kloo}.  Suppose that:
\begin{enumerate}
\item[(1)] the local component $\pi_v$ at our chosen {quasi-}split unramified place $v$ is generic.  
\item[(2)] there is a place $u \in S$ such that $\pi_u$ belongs to the discrete series.   
%\end{itemize}
\end{enumerate}
Then $\pi_w$ is tempered for all $w \notin S$ {at which $G$ is unramified}. 

{
More generally, with $w$ and $u$ as above, suppose $\pi_u$ is $\iota$-tempered.  Then $\pi_w$ is $\iota$-tempered for all $w \notin S$.  In particular, if $v = u$
and $\pi_v$ is generic and $\iota$-tempered then  $\pi_w$ is $\iota$-tempered for all $w \notin S$.}
\end{thm}

If $G$ is of type $C_n$ or $D_n$, we assume to begin with that  $G$ is adjoint, and  subsequently reduce to this case.   We may thus 
choose   representations $\tau_0$ of $\LGr$ so that the restriction to the Langlands dual group $G^\vee$
contains a representation with highest weight $\lambda^\vee$, determined by the following table:
 \begin{enumerate} 
\item[(i)] If $G=PGL(n)$, $n$ even, then $\lambda^\vee = \omega_1^\vee$, with the numbering of the roots in the Dynkin diagram:
\[\xymatrix{1\ar@{-}[r] &2 \ar@{-}[r] &3  \ar@{-}[r] &\dotsb \ar@{-}[r] &(n-1)};
\]
\item[(ii)] If $G=SO(2n+1)$, then $\lambda^\vee = \omega_n^\vee$ for the diagram
\[\xymatrix{1\ar@{<=}[r] &2 \ar@{-}[r] &3  \ar@{-}[r] &\dotsb \ar@{-}[r] &n};
\]
\item[(iii)] If $G=PSp(2n)$, then $\lambda^\vee = \omega_1^\vee$ for the diagram
\[\xymatrix{1\ar@{=>}[r] &2 \ar@{-}[r] &3  \ar@{-}[r] &\dotsb \ar@{-}[r] &n};
\]
\item[(iv)] If $G=PSO(2n)$, then we must take both $\lambda^\vee = \omega_1^\vee$ and $\lambda^\vee = \omega_n^\vee$ for the diagram
\[\xymatrix{1\ar@{-}[rd]\\ &3 \ar@{-}[r] \ar@{-}[ld] &4  \ar@{-}[r] &\dotsb \ar@{-}[r] &n;\\
2}
\]
\item[(v)] If $G=E_7$, then $\lambda^\vee = \omega_7^\vee$ for the diagram
\[\xymatrix{1\ar@{-}[r] &3 \ar@{-}[r]  &4\ar@{-}[d]  \ar@{-}[r] &5 \ar@{-}[r] &6 \ar@{-}[r] &7.\\
&&2}
\]
\end{enumerate}

%Notation is explained in {\color{red}  Perhaps add a brief explanation of the notation?}  
 The relevant results regarding the unitarity of local factors are based on the classification of generic unitary spherical representations, which is recalled in \S \ref{classificationsection} and \S \ref{s:nonsplit}.
  
  \medskip
  
 The following definition will play an important role in the proof.

\begin{defn}\label{mwf} Let  $w$ be an integer.  A  \emph{pure Weil-Deligne representation of weight $w$} of $K_u$ is a pair $(\varphi,N)$, where 
$$\varphi:  W_{K_u} \ra GL(V)$$
is a representation of the Weil group of the local field $K_u$ on an $m$-dimensional $\Qlb$-vector space $V$, and $N: V \ra V$ is a nilpotent operator, such that
the pair $(\varphi,N)$ satisfies the usual properties of a Weil-Deligne representation, and such that, for any choice of Frobenius element $Frob_u \in W_{K_u}$ we have
\begin{itemize}
\item[(i)]  The eigenvalues of $\varphi(Frob_u)$ are all $q$-numbers of integer weight.
\item[(ii)]  The subspace $W_aV \subset V$ of eigenvectors for $\varphi(Frob_u)$ with eigenvalues of weight $\leq a$ is invariant
under $(\varphi,N)$;
\item[(iii)]  Letting $gr_aV = W_aV/W_{a-1}V$ the map
$$N:  gr_{w-i}V \ra gr_{w+i}V$$
is an isomorphism for all $i$.
\end{itemize}
\end{defn}

 \begin{proof}  First assume $\pi_u$ is in the discrete series.  Recall that we are assuming, temporarily, that $G$ is adjoint
 if it is of type $C_n$ or $D_n$.  Let $\tau$ be either the adjoint representation of $\LGr$ or the representation $\tau_0$ in the above list.
By \cite[Th\'eor\`eme (3.4.1) (i)]{De} or \cite[Corollary VII.8]{Laf02}, the semisimple $\ell$-adic local system $\CL(\pi,\tau_0)$ can be written as a (finite) direct sum
\begin{equation}\label{chidecomp} \CL(\pi,\tau) = \oplus_{\chi/\sim} \CL(\pi,\tau)_\chi \otimes \chi
\end{equation}
where each $\CL(\pi,\tau)_\chi$ is a punctually $\iota$-mixed local system with integer weights
and $\chi$ is the pullback to $C \setminus |S|$ of a rank $1$ $\ell$-adic local system over $Spec(k_1)$, in other words a continuous $\ell$-adic character of $Gal(\bar{k}_1/k_1)$.     Moreover, this decomposition is unique, when $\chi$ is taken to run over the equivalence classes for the following relation:  $\chi$ and $\chi'$ are equivalent, written $\chi \sim \chi'$ in \eqref{chidecomp}, if $\chi'\cdot \chi^{-1}$ is pure of some (integer) weight.

Consider the monodromy representation $\CL(\pi,\tau)_\chi\otimes \chi$ at $u$.    Let
$\CL^{ss}(\pi_u)$ denote the semisimple $\ell$-adic local parameter attached to $\pi_u$ by Genestier and Lafforgue in \cite{GLa}.  
Each irreducible composition factor of $\CL(\pi,\tau)_\chi\otimes \chi$ is then a  constituent of  $\tau\circ \CL^{ss}(\pi_u)$.  
But it is proved in \cite{GHS} that $\tau\circ \CL^{ss}(\pi_u)$
extends uniquely to a {pure Weil-Deligne representation} of weight $0$ for any $\tau$, in the sense of Definition \ref{mwf}.
By the uniqueness of the decomposition \eqref{chidecomp}, this implies that $\CL(\pi,\tau)$ is itself mixed, and indeed is pure of weight $0$.  

It follows that every irreducible component of  the semisimple $\ell$-adic local system $\CL(\pi,\tau)$ is punctually $\iota$-pure, for every $\iota$.   Thus for any
 $w \notin S$, the eigenvalues of $\tau\circ \sigma(Frob_w)$ are Weil $q_w$-numbers, where $q_w$ is the order of the residue field at $w$.   This in turn implies that for any
 weight $\alpha$ of $\tau$, we have
 \begin{equation}\label{qnumbers} \langle \alpha,\nu_w(\sigma) \rangle \in \frac{1}{2}\ZZ.
 \end{equation}
 
In particular, this is the case when $w = v$, where $\pi_v$ is generic. 
Applying this when $\tau$ is the adjoint representation and the representation $\tau_0$ above, Corollary \ref{c:half-integral}(1) (when $G(K_v)$ is split) and Corollary \ref{c:unramified} (when $G(K_v)$ is quasi-split nonsplit)  then imply that $\nu_v(\sigma) = 0$. 
 It follows that, for every $\iota$, every irreducible component of $\CL(\pi,\tau)$ is punctually $\iota$-pure of weight $0$ at $v$, and therefore at every unramified place $w$.  Lemma \ref{temperedprin} then implies that $\pi_w$ is tempered at every unramified place $w$.

If $G$ is of type $C_n$ or $D_n$, we now drop the assumption that $G$ is adjoint.  At this point we would like to say that there is a finite (ramified) cover $C'$ of $C$, with function
field $K'$ such that
$\pi$ admits a base change $\pi'$ to a cuspidal representation of   $G(\ad_{K'})$ with trivial central character.  Then $\pi'$ is a constituent of the pullback
to $G(\ad_{K'})$ of a cuspidal representation of the adjoint quotient of $G$, and we can argue as before.  Since the stable twisted trace formula is not currently available over
function fields, the existence of such a $\pi'$ is not known.  Instead we argue directly on the side of the parameters.  As above, we assume $\pi$ belongs to the $\ell$-adic parameter
$\sigma:  Gal(K^{sep}/K) \ra \LGr(\Qlb)$.  Now the obstruction to lifting $\sigma$ to a homomorphism $\tilde{\sigma}:  Gal(K^{sep}/K) \ra \LGad(\Qlb)$ lies in
$H^2(Gal(K^{sep}/K),Z)$, where $Z$ is the (finite) center of $\LGad(\Qlb)$.  It follows easily that there is a finite extension $K'$ of $K$, as above, such that
the restriction $\sigma'$ to $Gal(K^{sep}/K')$ lifts to a homomorphism $\tilde{\sigma'}:  Gal(K^{sep}/K') \ra \LGad(\Qlb)$.  Let $C'$ be the finite cover
of $C$ with function field $K'$, and let $S'$ be the set of places of $C'$ where $\tilde{\sigma'}$ is ramified.  Now we can compose with $\tau_0$ in the above
list to obtain a local system $\CL(\pi,\tau_0)$ over $C' \setminus |S'|$.  If $u'$ is a place of $K'$ above $u$, then the argument above shows that every component of the
restriction of $\CL(\pi,\tau_0)$ to the decomposition group at $u'$ extends to a pure Weil-Deligne representation of weight $0$, and we conclude as before.\footnote{By a well-known Baire category argument, we know that the image of $\sigma$ lies in $\LGr(E)$ for some finite extension $E$ of $\mathbb{Q}_\ell$.  Johan de Jong has pointed out that this implies that there is an open subgroup of the image of $\sigma$ that lifts to  $\LGad(\Qlb)$, so there is no need even to mention
the obstruction in Galois $H^2$.  }

If now $\pi_u$ is $\iota$-tempered, the Frobenius eigenvalues of $\CL^{ss}(\pi_u)$, composed with $\iota$, have absolute values that are half-integral powers of $q_u$, where $q_u$ is the order
of the residue field at $u$.  The remainder of the argument follows as in the discrete series case.
\end{proof}

\begin{rmk*}  When $\pi_u$ is not supercuspidal, the proof in \cite{GHS} that $\tau\circ \CL^{ss}(\pi_u)$ is pure of weight $0$ is based on an original argument of 
Beuzart-Plessis based in turn on the Deligne-Kazhdan simple trace formula, which is currently the most useful version of the trace formula generally available over function fields.  Beuzart-Plessis's argument will be contained in the appendix to a second draft of \cite{GHS}.
\end{rmk*}

\begin{cor}\label{locgeneric}  Let $\pi$ be as in the statement of Theorem \ref{mainthm}.  Then $\pi_w$ is generic at all unramified places $w$  such that $G(K_w)$ is quasi-split.  In particular, if $G$ is a quasi-split group then $\pi$ is locally generic almost everywhere.
\end{cor}

\begin{proof}  It remains to record that if $\pi_w$ is spherical and tempered, then it is generic. This is well known, for example, for the general case, see \cite[Proposition 7.4]{Re}, the subtlety having to do with the isogeny class of $G(K_w)$. When $G(K_w)$ is split adjoint, the claim is easy, and for the benefit of the reader, we explain how it follows from \cite{KL}.  By the classification theorems of \cite{KL}, the smooth irreducible $G(K_w)$-representations with Iwahori fixed vectors are in one-to-one correspondence with  $G^\vee$-conjugacy classes of triples:
\[(s,n,\psi):\ s\in G^\vee \text{ semisimple},\ n\in \mathfrak g^\vee,\ \text{Ad}(s)= q n,
\]
and $\psi$ is an irreducible representation of ``Springer type" of the group of components of the centralizer of $s$ and $n$ in $G^\vee$, such that the center of $G^\vee$ acts by the identity in $\psi$. Notice that $n$ must necessarily be nilpotent. Let $\pi_w(s,n,\psi)$ denote the corresponding $G(K_w)$-representation. Let $\{\bar n,h,n\}$ be a Lie triple in $\mathfrak g^\vee$. We may arrange so that $s\in T^\vee$ (the maximal torus of $G^\vee$) and $h\in \mathfrak t^\vee_{\mathbb R}$. Set $s_0=s q^{-h/2}$ so that $\text{Ad}(s_0)n=n$. The Kazhdan-Lusztig correspondence has a number of properties including:
\begin{enumerate}
\item $\pi_w(s,n,\psi)$ is spherical if and only if $n=0$. Since the centralizer of $s$ in $G^\vee$ is connected, $G^\vee$ being simply connected, the representation $\psi$ is automatically trivial. Moreover, $s$ is the Satake parameter of $\pi_w$.
\item $\pi_w(s,n,\psi)$ is tempered if and only if $s_0$ is compact.
\end{enumerate}
In our case, $\pi_w$ is both spherical and tempered, therefore $\pi_w=\pi_w(s,0,\text{triv})$, where $s=s_0$ is compact. But then the $q$-eigenspace of $\text{Ad}(s)$ in $\mathfrak g^\vee$ is zero, so $\pi_w$ is the only irreducible representation with semisimple parameter $s$, hence it must be the full minimal spherical principal series with Satake parameter $s$. In particular, $\pi_w$ is generic.
\end{proof}

Shahidi's $L$-packet conjecture asserts that every tempered $L$-packet contains a globally generic member.
Since two members of a hypothetical $L$-packet are locally isomorphic almost everywhere, Corollary \ref{locgeneric} is consistent with Shahidi's conjecture.  

\begin{remark}  There are precise conjectures \cite[(8.3)]{LafICM} to the effect  that the following is always true: 

\begin{Expectation}\label{stack}  The irreducible components of the $\ell$-adic local systems $\CL(\pi,\tau) = \CL(\sigma,\tau)$ are realized in the total direct image $Rf_!(\Qlb)$ with compact support, where
$f:  Z \ra C \setminus |S'|$ is a morphism of algebraic stacks %{\color{red}  Deligne-Mumford stacks?} 
of finite type
for some finite set $S' \supset S$, when $\tau$ is the adjoint representation of $\LGr$ on $Lie(\hG)$ and also when $\tau = \tau_0$ (if $G$ is not both simply connected and adjoint).
\end{Expectation}

The stacks in question are precisely the Deligne-Mumford
moduli stacks of shtukas that figure in Lafforgue's constructions of Galois parameters.
Assuming this to be the case, it is well-known to experts that every irreducible component of  the semisimple $\ell$-adic local system $\CL(\pi,\tau)$ is punctually $\iota$-pure with integer weights.  However, there seems to be no accessible proof in the literature.  We therefore provide a sketch of a proof in  \S \ref{purity}.  
Admitting the purity for the moment we see that, for any
 $w \notin S(U)$, the eigenvalues of $\tau\circ \sigma(Frob_w)$ are Weil $q_w$-numbers, where $q_w$ is the order of the residue field at $w$.   This implies as above that  any
 weight $\alpha$ of $\tau$ satisfies \eqref{qnumbers}.  In particular, this is the case when $w = v$, where $\pi_v$ is generic.
Reasoning as in the proof of Theorem \ref{mainthm}, Corollary \ref{c:half-integral}  then implies that $\nu_v(\sigma) = 0$. 

 It follows that every irreducible component of $\CL(\pi,\tau)$ is punctually $\iota$-pure of weight $0$ at $v$, and therefore at every unramified place $w$.  Lemma \ref{temperedprin} then implies that $\pi_w$ is tempered at every unramified place $w$.
 
 Lafforgue and Zhu have proved Expectation \ref{stack} to be true when the global parameter of $\pi$ is {\it elliptic}
 \cite[Proposition 1.2]{LZ}.   Presumably any $\pi$ satisfying the hypotheses of Theorem \ref{mainthm} is elliptic.  For elliptic parameters, howver, Theorem \ref{mainthm} is superfluous; it is explained in \S 3 of \cite{LZ} that such $\pi$ are already tempered at 
 unramified places.  
\end{remark}
  
% \subsubsection*{Sketch of the purity claim}\label{purity}  

\begin{sketch}[purity]\label{purity} We thank Martin Olsson for the following sketch of the proof of the purity claim.  It suffices in fact to show that the cohomology $Rf_*(\Qlb)$ is $\iota$-mixed.   The main reference is Theorem 2.11
 of the paper \cite{Su}, which proves that the property of being $\iota$-mixed is preserved under the six operations of Laszlo and Olsson (the stack version of Grothendieck's six operations) and Verdier duality.  We thus need to show that $Rf_!$ preserves the property that the weights are integral.   Now one can show that Verdier duality preserves the integrality by using a smooth cover to reduce to the case of schemes, so we may replace $Rf_!$ by $Rf_*$.  As in the argument in \cite{Su}, one reduces to the case when the target stack is a scheme. 
  Thus it suffices to consider the case $f: X \ra Y$, with $X$ an algebraic stack and $Y$.  By using a hypercover of $X$ by schemes and the associated spectral sequence we are thus reduced to the case of schemes.  By Verdier duality again, this reduces to the case of $Rf_!$, which is proved in \cite{De}.
\end{sketch}

\begin{remark}  The proof of Theorem \ref{mainthm} implies, by the arguments in \cite{GHS}, that the semisimple
Genestier-Lafforgue parameters $\CL^{\rm ss}(\pi_u)$ of the local components $\pi_u$ of $\pi$ at all places, including $u \in S$, extend (uniquely) to tempered Weil-Deligne parameters, in the sense defined in \cite{GHS}.  But this does not imply that the ramified components $\pi_u$ themselves are necessarily tempered, although it is generally believed that they must be.

\end{remark}

\section{Half-integral generic spherical unitary points}\label{classificationsection}

We fix notation. Let $k$ be a nonarchimedean local field with discrete valuation $\val_k$, ring of integers $\mathfrak O$, and finite residue field $\bF_q$. Let $G$ be a quasi-simple split $k$-group of adjoint type and $K=G(\mathfrak O)$, the maximal compact hyperspecial subgroup of $G(k)$. Let $B\supset T$ denote a $k$-rational Borel subgroup and a maximal $k$-split torus, respectively. Denote the corresponding based root datum of $(G,B,T)$ by $(X,\Phi, X^\vee,\Phi^\vee,\Pi)$, where $\Pi$ are the simple roots, and let $W$ be the finite Weyl group. Let $\Phi^+\supset\Pi$ denote the positive roots, and $\Phi^{\vee,+}$ the corresponding positive coroots. Let $\langle ~,~\rangle$ be the natural pairing between $X$ and $X^\vee$. Let $G^\vee$ be the complex Langlands dual group with maximal torus $T^\vee=X\otimes_\bZ \bC^\times$. Recall the polar decomposition $T^\vee=T^\vee_\bR\cdot T^\vee_c$, where $T^\vee_\bR=X\otimes_\bZ \bR_{>0}$, $T^\vee_c=X\otimes_\bZ S^1$. 

If $s\in T^\vee\simeq\Hom(X^\vee,\bC^\times)$, let $\chi_s\in X$ denote the unramified character of $T(k)$,
\begin{equation}
\chi_s=s\circ \val_T: T(k)\to \bC^\times,
\end{equation} 
where  $\val_T:T(k)\to X^\vee$ is 
defined by 
\[\langle \val_T(t),\lambda\rangle=\val_k(\lambda(t)),\quad \text{for all }\lambda\in X\text{ and all }t\in T(k).
\]
Regard $\chi_s$ as a character of $B(k)$ by pullback from $T(k)$ and let
\[X(s)=i_{B(k)}^{G(k)}(\chi_s)
\]
be the (unitarily) induced spherical minimal principal series. The representation $X(s)$ has a unique $K$-spherical quotient $L(s)$, and there is a one-to-one correspondence between irreducible $K$-spherical $G(k)$-representations and $W$-orbits in $T^\vee$:
\[s\leftrightarrow L(s).
\]
The semisimple element $s$ (or rather its $W$-orbit in $T^\vee$) is called the \emph{Satake parameter} of the irreducible $K$-spherical representation $L(s)$. 

Let $\nu\in \ft^\vee_\bR=\text{Lie}(T^\vee_\bR)$ and write
\[s_\nu=\exp(\nu\log(q))\in T^\vee_\bR.
\]
Following \cite{BM}, one may refer to these parameters as \emph{real Satake parameters}. Every Satake parameter $s\in T^\vee$ can be written uniquely as
\begin{equation}\label{e:polar}
s=s_c\cdot s_\nu,\text{ for some } s_c\in T^\vee_c\text{ and }\nu\in  \ft^\vee_\bR.
\end{equation}
For simplicity, we also write
\[X(\nu):=X(s_\nu),\quad \nu \in \ft^\vee_\bR,
\]
and $L(\nu)$ for its $K$-spherical subquotient. With this notation, the trivial $G(k)$-representation is $L(\rho)$, where $\rho=\frac 12\sum_{\alpha\in \Phi^+}\alpha\in \ft^\vee_\bR$.

The irreducible unitary spherical $G(k)$-representations with real Satake parameters are classified in \cite{Ba} for symplectic and orthogonal split groups, \cite{Ci-F4,Ci-E6,BC-E8} for exceptional split groups, via solving the equivalent corresponding problem for Iwahori-Hecke algebras \cite{BM}. The case $G=GL(n)$ is by now well known and it is due to Tadi\' c \cite{Ta}.

We will only be interested in the cases when $X(s)=L(s)$, that is the case of $K$-spherical representations which are also generic (in the sense of admitting Whittaker vectors). It is well known that when $\nu \in  \ft^\vee_\bR$,
\[X(\nu)=L(\nu) \text{ if and only if } \langle\al^\vee,\nu\rangle\neq 1\text{ for all }\al^\vee\in \Phi^{\vee,+}.
\]

Denote
\begin{equation}
\begin{aligned}
\CC_0&=\{\nu \in  \ft^\vee_\bR\mid \langle \al^\vee,\nu\rangle\ge 0,\text{ for all }\al\in \Pi\},\\
\quad \CU^g_0&=\{\nu\in \CC^\vee_0\mid   \langle\al^\vee,\nu\rangle\neq 1\text{ for all }\al^\vee\in \Phi^{\vee,+}, \text{ and }L(\nu)\text{ is unitary}\}.
\end{aligned}
\end{equation}
Again it is well known that for $L(\nu)$ to admit a nonzero invariant Hermitian form it is necessary and sufficient that 
\begin{equation}
w_0(\nu)=-\nu,\text{ where } w_0 \text{ is the longest Weyl group element}.
\end{equation}
This is automatic when $\Phi^\vee$ is of types $B,C, G_2, F_4, E_7, E_8$. Denote 
\[\CC_{0,h}=\{\nu\in \CC^\vee_0\mid w_0(\nu)=-\nu\},
\]
so that $\CU^g_0\subset \CC_{0,h}.$

The explicit description of $\CU^g_0$ is given in {\it loc. cit.}. The nature of the answer is the following. The arrangement of hyperplanes
\[\alpha^\vee=1,\quad \alpha^\vee\in \Phi^{\vee,+},
\]
partitions the  fundamental Weyl chamber $\CC_0$ into $m$ open regions, 
\[m=\prod_{i=1}^\ell \frac{d_i+h}{d_i},\]
where $h$ is the Coxeter number, $\ell$ is the rank, and $d_i$ are the fundamental degrees for the \emph{coroot system} $\Phi^\vee$. For example, for $F_4$ there are $105$, while for $E_8$, there are $25080$ such open regions. It is clear that $\CU^g_0$ is a union of open regions in this arrangement of hyperplanes intersected with $\CC_{0,h}$. 

For each simple root system, denote by $\al_i$, $\al_i^\vee$, $\omega_i$, $i=1,\dots,\ell$ the simple roots, simple coroots, and fundamental weights, respectively, and
\[\rho=\sum_{i=1}^\ell \om_i.
\]
Consider the poset of positive coroots. We say that a positive coroot $\al^\vee$ has level $r$ if $\langle \al^\vee,\rho\rangle=r$. Let $\gamma^\vee$ be the highest positive coroot, the unique coroot of level $h-1$. Recall that an \emph{alcove} is a connected (open) component in 
\[
\ft^\vee_\bR\setminus \bigcup_{\alpha^\vee\in\Phi^\vee, m\in \bZ_{\ge 0}} \{\nu\in \ft^\vee_\bR \mid \langle\alpha^\vee,\nu\rangle=m\}.
\]
In particular, the \emph{fundamental alcove} is
\begin{equation}
\CA_0=\{\nu\in \CC_0\mid \langle \gamma^\vee,\nu\rangle<1\}.
\end{equation}

By a well-known deformation argument (\emph{complementary series}), it is easy to see that:
\begin{lemma}
$\CA_0\cap \CC_{0,h}\subseteq \CU^g_0.$
\end{lemma}

Another classical result says that no unbounded region can be unitary, see \cite{HM} (or \cite[\S3.3]{BC-gen} for a proof in this setting):

\begin{lemma}\label{l:rho}
If $\nu\in \CC_0$ is such that $\langle\al^\vee_i,\nu\rangle>1$ for a simple coroot $\al^\vee_i$, then $\nu\notin \CU^g_0.$
\end{lemma}

But beyond these classical results, the determination of $\CU^g_0$ is difficult and has been achieved via a case-by-case analysis. We will use some elements of the explicit description which will be recorded in the sequel, but here let us give the general form of the answer.

\begin{thm}[\cite{Ba,Ci-F4,Ci-E6,BC-E8}]\label{t:alcoves} Suppose $G$ is a quasi-simple split group of adjoint type. The set $\CU_0^g$ is the intersection of the hermitian locus $\CC_0$ with a union of $2^d$ alcoves in $\CC_0$ (including the fundamental alcove $\CA_0$), where
\begin{itemize}
\item  $d=0$, if $G=PGL(n)$ or $SO(2n+1)$;
\item $d=\lfloor\frac {n-1}2\rfloor$ is $G=PSp(2n)$;
\item $d=n-1$ if $G=PSO(4n)$, $n\ge 2$, or if $G=PSO(4n+2)$, $n\ge 1$;
\item $d=1$, if $G=G_2,$ $F_4$, or $E_6$;
\item $d=3$, if $G=E_7$;
\item $d=4$, if $G=E_8$.
\end{itemize}
\end{thm}

\subsection{Integral unitary points for the adjoint $G^\vee$-representation}

We prove first that there are no nonzero generic unitary parameters $\nu$ that are integral with respect to the adjoint $G^\vee$-representation. 

\begin{prop}\label{p:integral-unitary}
If $\nu\in\CC_0$ is such that $\langle\alpha^\vee,\nu\rangle\in \bZ_{>0}$ for a positive coroot $\alpha^\vee$, then $\nu\notin \CU^g_0.$
\end{prop}

\begin{proof}
Let $\nu\in\CC_{0,h}$, otherwise it can't be a unitary parameter since it is not hermitian. If $\langle\alpha^\vee,\nu\rangle=m$ for a positive coroot $\alpha^\vee$ and a positive integer $m$, then by definition, $\nu$ does not belong to any alcove in $\CC_0$. Since $\CU_0^g$ is a union of alcoves by Theorem \ref{t:alcoves}, $\nu\notin \CU_0^g$.
\end{proof}

\subsection{Half-integral unitary points for the adjoint $G^\vee$-representation}

We would like to determine the generic unitary parameters $\nu$ that are also $\epsilon$-half-integral with respect to the adjoint $G^\vee$-representation for a fixed real number $0\le\epsilon<\frac 12$. In other words, we need to calculate:
\begin{equation}
\CU^{g,\frac 12+\epsilon}_0=\CU^g_0\cap \CC_0^{\frac 12+\epsilon},\text{ where }\CC_0^{\frac 12+\epsilon}=\{\nu\in \CC_0\mid \langle \al^\vee,\nu\rangle\in \epsilon+\frac 12\bZ,\text{ for all }\al^\vee\in \Phi^{\vee}\}.
\end{equation}
Since the condition has to be satisfied for both $\al^\vee$ and $-\al^\vee$, it follows immediately that $2\epsilon\in \frac 12\bZ$, hence $\CC_0^{\frac 12+\epsilon}=\emptyset$ unless 
\[\epsilon\in \{0,\frac 14\}.
\]
If the $\Phi$ contains a simple root subsystem of rank at least $2$ (in particular, if $G$ is quasi-simple and $G\neq PGL(2)$), then there exist simple coroots $\al_1^\vee$ and $\al_2^\vee$ such that $\beta^\vee=\al_1^\vee+\al_2^\vee$ is also a coroot. Then 
\[\langle\beta^\vee,\nu\rangle=\langle \al_1^\vee,\nu\rangle+\langle \al_2^\vee,\nu\rangle\in 2\epsilon+\frac 12\bZ.
\]
On the other hand, $\langle\beta^\vee,\nu\rangle\in \epsilon+\frac 12\bZ$. From this, $\epsilon\in\frac 12\bZ$ and so $\epsilon=0$ necessarily. Hence, if $G$ is quasi-simple of rank at least $2$,
\begin{equation}
\CC^{\frac 12+\epsilon}_0=\emptyset\text{ unless }\epsilon=0.
\end{equation}
For $G=PGL(2)$, we will only be interested in $\epsilon=0$ as well. 

\medskip

We now consider the case $\ep=0$. The ordering of the roots and the explicit coordinates are given in the subsections, for exceptional groups they are the Bourbaki ones; the coordinates are such that the pairings are just the dot product of vectors. Write a point $\nu\in \CC_0$ in the \emph{fundamental weight coordinates}: 
\[\nu=\sum_{i=1}^\ell \nu_i\omega_i,\quad \nu_i\ge 0,\text{ for all }1\le i\le \ell.\]
Then $\nu\in\CC_0^{\frac 12}$ if and only if $\nu_i\in \frac 12\bZ_{\ge 0}$ for all $i$. By Lemma \ref{l:rho}, it follows that a necessary condition for $\nu\in \CU^{g,\frac 12}_0$ is
\begin{equation}\label{e:a-i}
\nu_i\in \{0,\frac 12\},\text{ for all } 1\le i\le \ell.
\end{equation}

\begin{thm}\label{t:half-adjoint}
Suppose $G$ is a quasi-simple $k$-split group of adjoint type. The set $\CU^{g,\frac 12}_0$ is as follows:
\begin{itemize}
\item $G=PGL(n)$, $n$ even, $\CU^{g,\frac 12}_0=\{0,\frac 12\omega_{\frac n2}\}$;
\item $G=SO(2n+1)$,  $\CU^{g,\frac 12}_0=\{0,\frac 12\omega_1\}$;
\item $G=PSp(2n)$,  $\CU^{g,\frac 12}_0=\{0,\frac 12\omega_n\}$;
\item $G=PSO(2n)$, $n$ even,  $\CU^{g,\frac 12}_0=\{0,\frac 12\omega_1,\frac 12\omega_2,\frac 12\omega_n\}$;
\item $G=PSO(2n)$, $n$ odd,  $\CU^{g,\frac 12}_0=\{0,\frac 12\omega_n\}$;
 \item $G=E_7$, $\CU^{g,\frac 12}_0=\{0,\frac 12\omega_7\}$;
 \item For all other cases, $\CU^{g,\frac 12}_0=\{0\}$.
 \end{itemize}
\end{thm}

\begin{proof}
The strategy for each simple root system is the following:

\begin{enumerate}
\item The classification of the generic spherical unitary dual implies, in particular, that there exists a maximum level $r_0\ge 1$ such that $\nu\in \CU^g_0$ \emph{only if} $\langle \al^\vee,\nu\rangle<1$ for \emph{all} $\al^\vee$ of level $r_0$. The fact that $r_0\ge 1$ is equivalent with Lemma \ref{l:rho}.

Record which of these points satisfy $\langle \al^\vee,\nu\rangle<1$ for  \emph{all} $\al^\vee$ of level $r_0$.
\item For each one of the resulting points, check if it belongs to one of the hyperplanes of reducibility. In all cases, except in the exceptions listed in the theorem, the nonzero points surviving step (3) lie on the hyperplane $\gamma^\vee=1$, where $\gamma^\vee$ is the highest coroot. Otherwise, verify that the remaining point(s) lie in the one of the unitary regions of $\CU^g_0$; it turns out they are always in the fundamental alcove $\gamma^\vee<1$.
\end{enumerate}
The details for each simple root system are presented in the next subsections. 
\end{proof}

\subsection{$PGL(n)$} For type $A$, it is more convenient to use standard coordinates, rather than coordinates with respect to the fundamental weights. The positive roots are $\ep_i-\ep_j$, $1\le j\le i\le n$, and we write a typical dominant hermitian parameter $\nu\in \CC_0$ as:
\begin{align*}
\nu&=(a_1,a_2,\dots,a_k,-a_k,-a_{k-1},\dots,-a_1), &\text{ if }n&=2k,\\
\nu&=(a_1,a_2,\dots,a_k,0,-a_k,-a_{k-1},\dots,-a_1), &\text{ if }n&=2k+1,
\end{align*}
where $a_1\ge a_2\ge\dots\ge a_k\ge 0.$ The unitarity condition is \cite{Ta}:
\begin{equation}\label{bound-A}
0\le a_k\le a_{k-1}\le \dots\le a_2\le a_1< \frac 12.
\end{equation}
Suppose $n=2k$ (even). Then $\nu$ is half-integral if and only if $a_k\in \frac 14\bZ$ and $\nu_i-\nu_{i+1}\in \frac 12\bZ$, $1\le i\le k-1$. Hence the only possible nonzero point in the unitary region is
\[\frac 12\om_{k}, \text{ for which }\nu_1=\nu_2=\dots=\nu_k=\frac 14.
\]
This point is in the unitary region obviously.

\smallskip

If $n=2k+1$ (odd), then $\nu$ is half-integral if and only if $a_i\in \frac 12\bZ$, but no such point is in the unitary region (\ref{bound-A}).

\subsection{$SO(2n+1)$} In the coordinates $\ep_1,\dots,\ep_n$ in $\bR^n$,  we use the coroots 
\[\al_1^\vee=2\ep_1,\ \al_2^\vee=-\ep_1+\ep_2,\  \al_3^\vee=-\ep_2+\ep_3,\dots,\  \al_n^\vee=-\ep_{n-1}+\ep_n,
\]
and weights
\[\om_1=\frac 12(\ep_1+\ep_2+\dots+\ep_n),\quad \om_i=\ep_i+\ep_{i+1}+\dots+\ep_n,\ 2\le i\le n.
\]
The highest coroot is $\gamma^\vee=2\ep_n$. From \cite[Theorem 3.1]{Ba}, $r_0=2n-1$ and $\gamma^\vee$ is the only coroot of that level. If we write $\nu=\sum_{i=1}^n\nu_i\om_i \in\CC_0^{\frac 12}$, $\nu_1,\nu_2\in \frac 12\bZ_{\ge 0}$, then the unitary bound implies
\[\nu_1+2\nu_2+2\nu_3+\dots+2\nu_n<1,
\]
hence $\nu_2=\nu_3=\dots=\nu_n=0$. We have two points left: $0$ and $\frac 12\om_1$. Since $\langle \gamma^\vee,\frac 12\om_1\rangle=\frac 12<1$, this point is in the fundamental alcove, hence unitary.

\subsection{$PSp(2n)$}\label{sub:psp}  In the coordinates $\ep_1,\dots,\ep_n$ in $\bR^n$,  we use the coroots 
\[\al_1^\vee=\ep_1,\ \al_2^\vee=-\ep_1+\ep_2,\  \al_3^\vee=-\ep_2+\ep_3,\dots,\  \al_n^\vee=-\ep_{n-1}+\ep_n,
\]
and weights
\[\om_i=\ep_i+\ep_{i+1}+\dots+\ep_n,\ 1\le i\le n.
\]
The highest coroot is $\gamma^\vee=\ep_{n-1}+\ep_n$. From \cite[Theorem 3.1]{Ba}, we can deduce that 
\[r_0=\begin{cases} n,& n \text{ odd},\\n+1, &n \text{ even}.\end{cases}
\]
If $n$ is odd, the coroots at level $r_0$ are
\begin{equation}
\ep_1+\ep_{n-1},\ \ep_2+\ep_{n-2},\dots,\ \ep_{\frac {n-1}2}+\ep_{\frac {n+1}2},\ \ep_n,
\end{equation}
while if $n$ is even, they are
\begin{equation}
\ep_1+\ep_{n},\ \ep_2+\ep_{n-1},\dots,\ \ep_{\frac {n}2}+\ep_{\frac {n}2+1}.
\end{equation}
Write $\nu=\sum_{i=1}^n\nu_i\om_i \CC_0^{\frac 12}$, $\nu_1,\nu_2\in \frac 12\bZ_{\ge 0}$. Then these unitary bounds imply, when $n$ is odd:
\begin{align*}
2\nu_1+\nu_2+\nu_3+\dots+\nu_{n-1}&<1,\\
2\nu_1+2\nu_2+\nu_3+\dots+\nu_{n-2}&<1,\\
\vdots\\
2\nu_1+2\nu_2+\dots+2\nu_{\frac{n-1}2}+\nu_{\frac{n+1}2}&<1,\\
\nu_1+\nu_2+\dots+\nu_n&<1,
\end{align*}
and when $n$ is even:
\begin{align*}
2\nu_1+\nu_2+\nu_3+\dots+\nu_{n}&<1,\\
2\nu_1+2\nu_2+\nu_3+\dots+\nu_{n-1}&<1,\\
\vdots\\
2\nu_1+2\nu_2+\dots+2\nu_{\frac{n}2}+\nu_{\frac{n}2+1}&<1.\\
\end{align*}
In both cases, it follows immediately that
\[\nu_1=\nu_2=\dots=\nu_{\lfloor\frac n2\rfloor}=0,
\]
and
\[\nu_{\lfloor\frac n2\rfloor+1}+\nu_{\lfloor\frac n2\rfloor+2}+\dots+\nu_n<1.
\]
This means that at most one $\nu_j=\frac 12$, $j=\lfloor\frac n2\rfloor+1,\dots,n$, and the others are $0$. This gives the half-integral points $0$ and
\[\frac 12\om_j,\quad j=\lfloor\frac n2\rfloor+1,\dots,n.
\]
But for $j<n$, $\frac 12\om_j$ lies on the hyperplane $\gamma^\vee=1$, so the only nonzero point left is $\frac 12\om_n$. Since $\langle \gamma^\vee,\frac 12\om_n\rangle=\frac 12<1$, this point is in the fundamental alcove, hence unitary.

\subsection{$PSO(2n)$} In the coordinates $\ep_1,\dots,\ep_n$ in $\bR^n$,  we use the coroots 
\[\al_1^\vee=\ep_1+\ep_2,\ \al_2^\vee=-\ep_1+\ep_2,\  \al_3^\vee=-\ep_2+\ep_3,\dots,\  \al_n^\vee=-\ep_{n-1}+\ep_n,
\]
and weights
\[\om_1=\frac 12(\ep_1+\ep_2+\dots+\ep_n),\ \om_2=\frac 12(-\ep_1+\ep_2+\dots+\ep_n),\quad \om_i=\ep_i+\ep_{i+1}+\dots+\ep_n,\ 3\le i\le n.\]
The highest coroot is $\gamma^\vee=\ep_{n-1}+\ep_n$. Write $\nu=\sum_{i=1}^n\nu_i\om_i \CC_0^{\frac 12}$, $\nu_1,\nu_2\in \frac 12\bZ_{\ge 0}$. If $n$ is odd, we need to impose in addition that $\nu_1=\nu_2$ in order for this parameter to correspond to a hermitian spherical representation.

Suppose $n$ is even. From \cite[Theorem 3.1]{Ba}, we can deduce that $r_0=n-1$ and the coroots at this level are
\begin{equation}
\ep_1+\ep_n,\ -\ep_1+\ep_n,\ \ep_2+\ep_{n-1},\ \ep_3+\ep_{n-2},\dots,\ \ep_{\frac n2}+\ep_{\frac n2+1}.
\end{equation}
The unitary bound given by these coroots implies (via a similar discussion to that for $Sp(2n)$, $n$ even), that
\[\nu_3=\nu_4=\dots=\nu_{\frac n2}=0,
\]
and at most one of $\nu_1,\nu_2,\nu_j$, $j=\frac n2+1,\dots,n$, can equal $\frac 12$, the rest being zero. In addition to the origin $0$, we are left with
\[\frac 12\om_1,\frac 12\om_2,\quad \frac 12\om_j,\ j=\frac n2+1,\dots,n.
\]
All $\frac 12\om_j$, $\frac n2+1\le j\le n-1$ lie on the hyperplane $\gamma^\vee=1$, so the only nonzero points remaining are $\frac 12\om_1,\frac 12\om_2,\frac 12\om_n$, for which we have
\[\langle \gamma^\vee,\frac 12\om_1\rangle=\langle \gamma^\vee,\frac 12\om_2\rangle=\langle \gamma^\vee,\frac 12\om_n\rangle=\frac 12<1,
\]
meaning they are in the fundamental alcove and hence unitary.

\smallskip

If $n$ is odd, a similar analysis holds, except that now $\frac 12\om_1$ and $\frac 12\om_2$ are not hermitian.

\subsection{$G_2$} In Bourbaki coordinates, realize the root system of type $G_2$ in the hyperplane perpendicular to the vector $(1,1,1)$ in $\bR^3$:
\begin{equation}
\begin{aligned}
\al_1&=(2,-1,-1) & \al_1^\vee&=(\frac 23,-\frac 13,-\frac 13) &\omega_1&=(1,1,-2)\\
\al_2&=(-1,1,0) & \al_2 ^\vee&=(-1,1,0) &\omega_2&=(0,1,-1).
\end{aligned}
\end{equation}
Write $\nu=\nu_1\omega_1+\nu_2\omega_2\in \CC_0^{\frac 12}$, $\nu_1,\nu_2\in \frac 12\bZ_{\ge 0}$. From \cite[Appendix B]{Ci-F4}, we see that $r_0=3$ and there is only one coroot $\beta^\vee=2\al_1^\vee+\alpha_2^\vee$ at this level. Hence a necessary condition for unitarity is
\begin{equation}
2\nu_1+\nu_2<1.
\end{equation}
From this, $2\nu_1<1$, so necessarily $\nu_1=0$. Then $\nu_2<1$, so $\nu_2=0$ or $\nu_2=\frac 12$. If $\nu_2=\frac 12$, we get the point $\nu=\frac 12\omega_2$, but this lies on the hyperplane $\gamma^\vee=1$, where $\gamma^\vee=3\alpha_1^\vee+2\alpha_2^\vee$ is the longest coroot.

\subsection{$F_4$} In coordinates in $\bR^4$:
\begin{equation}
\begin{aligned}
\al_1&=\frac 12(1,-1,-1,-1) &\al_1^\vee&=(1,-1,-1,-1) &\omega_1&=(1,0,0,0)\\
\al_2&=(0,0,0,1) &\al_2^\vee&=(0,0,0,2) &\omega_2&=(\frac 32,\frac 12,\frac 12,\frac 12)\\
\al_3&=(0,0,1,-1) &\al_3^\vee&=(0,0,1,-1) &\omega_3&=(2,1,1,0)\\
\al_4&=(0,1,-1,0) &\al_4^\vee&=(0,1,-1,0) &\omega_4&=(1,1,0,0).
\end{aligned}
\end{equation}
%Hence $\rho=(\frac {11}2,\frac 52,\frac 32,\frac 12)$. 
From \cite[Table, p. 126]{Ci-F4}, we see that $r_0=9$ and there is only one coroot $\beta^\vee=(1,1,1,-1)=\al_1^\vee+2\al_2^\vee+4\al_3^\vee+2\al_4^\vee$ of level $9$.

Write $\nu=\sum_{i=1}^4\nu_i\omega_i\in \CC_0^{\frac 12}$, $\nu_i\in \frac 12\bZ_{\ge 0}$ for all $i$. Using $\langle\beta^\vee,\nu\rangle<1$, we get
\begin{equation}
\nu_1+2\nu_2+4\nu_3+2\nu_4<1.
\end{equation}
Immediately, $\nu_2=\nu_3=\nu_4=0$ and $\nu_1<1$. So either $\nu_1=\frac 12$ or $\nu_1=0$. In the first case, we get the point $\nu=\frac 12\omega_1$, but this lies on the hyperplane $\gamma^\vee=1$, where $\gamma^\vee=(2,0,0,0)$ is the longest coroot. Thus we are left with the origin.

\subsection{$E_8$} In coordinates in $\bR^8$ with respect to a basis $\ep_i$, $1\le i\le 8$:
\begin{equation}
\begin{aligned}
\al_1^\vee&=\frac 12(1,-1,-1,-1,-1,-1,-1,1), &\al_2^\vee&=\ep_1+\ep_2,& \al_3^\vee&=-\ep_1+\ep_2, & \al_4^\vee&=-\ep_2+\ep_3, \\
 \al_5^\vee&=-\ep_3+\ep_4,  &\al_6^\vee&=-\ep_4+\ep_5,&  \al_7^\vee&=-\ep_5+\ep_6,  &\al_8^\vee&=-\ep_6+\ep_7;\\
\end{aligned}
\end{equation}
\begin{equation}
\begin{aligned}
\om_1&=2\ep_8, &\om_2&=(\frac 12,\frac 12,\frac 12,\frac 12,\frac 12,\frac 12,\frac 12,\frac 52), &\om_3&=(-\frac 12,\frac 12,\frac 12,\frac 12,\frac 12,\frac 12,\frac 12,\frac 72),\\
 \om_4&=(0,0,1,1,1,1,1,5),& \om_5&=(0,0,0,1,1,1,1,4), &\om_6&=(0,0,0,0,1,1,1,3),\\
 \om_7&=(0,0,0,0,0,1,1,2), &\om_8&=\ep_7+\ep_8.
\end{aligned}
\end{equation}
From \cite[\S7.2.3]{BC-E8}, we see that $r_0=15$ and there are four coroots at level $15$. They are:
\begin{equation}\label{bound-E8}
\begin{aligned}
\beta_1^\vee&=\al_1^\vee+2\al_2^\vee+2\al_3^\vee+4\al_4^\vee+3\al_5^\vee+2\al_6^\vee+\al_7^\vee=\frac 12(1,-1,1,1,1,1,-1,1) \\
\beta_2^\vee&=\al_1^\vee+2\al_2^\vee+2\al_3^\vee+3\al_4^\vee+3\al_5^\vee+2\al_6^\vee+\al_7^\vee+\al_8^\vee=\frac 12(1,1,-1,1,1,-1,1,1) \\
\beta_3^\vee&=\al_1^\vee+2\al_2^\vee+2\al_3^\vee+3\al_4^\vee+2\al_5^\vee+2\al_6^\vee+2\al_7^\vee+\al_8^\vee=\frac 12(1,1,1,-1,-1,1,1,1) \\
\beta_4^\vee&=\al_1^\vee+\al_2^\vee+2\al_3^\vee+3\al_4^\vee+3\al_5^\vee+2\al_6^\vee+2\al_7^\vee+\al_8^\vee=\frac 12(-1,-1,-1,1,-1,1,1,1). \\
\end{aligned}
\end{equation}
These are the coroots labelled $89,90,91,92$ in {\it loc. cit.}.

Write $\nu=\sum_{i=1}^8\nu_i\omega_i\in \CC_0^{\frac 12}$, $\nu_i\in \frac 12\bZ_{\ge 0}$ for all $i$. Using (\ref{bound-E8}), we immediately see that necessary conditions are
\[\ 2\nu_2<1,\ 2\nu_3<1,\ 4\nu_4<1,\ 3\nu_5<1,\ 2\nu_6<1,\ 2\nu_7<1,\text{ and } \nu_1+\nu_8<1.
\]
From this, we must have 
\[\nu_2=\nu_3=\nu_4=\nu_5=\nu_6=\nu_7=0,\text{ and }\nu_1,\nu_8\in\{0,\frac 12\}\text{ but not }\nu_1=\nu_8=\frac 12.
\]
This means that there are three remaining points to check:
\begin{itemize}
\item $\nu=0$, this is the origin;
\item $\nu=\frac 12\omega_1$;
\item $\nu=\frac 12\omega_8$;
\end{itemize}
But the last two points are both on the hyperplane $\gamma^\vee=1$, where $\gamma^\vee=\ep_7+\ep_8$ is the highest coroot.

\subsection{$E_7$} We use $\al_i^\vee$, $1\le i\le 7$ from $E_8$, but the coordinates for the fundamental weights are different:
\begin{equation}
\begin{aligned}
\om_1&=(0,0,0,0,0,0,-1,1) &\om_2&=(\frac 12,\frac 12,\frac 12,\frac 12,\frac 12,\frac 12,-1,1), &\om_3&=(-\frac 12,\frac 12,\frac 12,\frac 12,\frac 12,\frac 12,-\frac 32,\frac 32),\\
 \om_4&=(0,0,1,1,1,1,-2,2),& \om_5&=(0,0,0,1,1,1,-\frac 32,\frac 32), &\om_6&=(0,0,0,0,1,1,-1,1),\\
 \om_7&=(0,0,0,0,0,1,-\frac 12,\frac 12).
\end{aligned}
\end{equation}
From \cite[\S7.2.2]{BC-E8}, we see that $r_0=9$ and there are four coroots at level $9$. They are:
\begin{equation}\label{bound-E7}
\begin{aligned}
\beta_1^\vee&=\al_1^\vee+\al_2^\vee+2\al_3^\vee+2\al_4^\vee+2\al_5^\vee+\al_6^\vee=\frac 12(-1,1,-1,1,1,-1,-1,1) \\
\beta_2^\vee&=\al_1^\vee+\al_2^\vee+2\al_3^\vee+2\al_4^\vee+\al_5^\vee+\al_6^\vee+\al_7^\vee=\frac 12(-1,1,1,-1,-1,1,-1,1) \\
\beta_3^\vee&=\al_1^\vee+\al_2^\vee+\al_3^\vee+2\al_4^\vee+2\al_5^\vee+\al_6^\vee+\al_7^\vee=\frac 12(1,-1,-1,1,-1,1,-1,1) \\
\beta_4^\vee&=\al_2^\vee+\al_3^\vee+2\al_4^\vee+2\al_5^\vee+2\al_6^\vee+\al_7^\vee=\ep_5+\ep_6. \\
\end{aligned}
\end{equation}
These are the coroots labelled $46,47,48,49$ in {\it loc. cit.}.

Write $\nu=\sum_{i=1}^7\nu_i\omega_i\in \CC_0^{\frac 12}$, $\nu_i\in \frac 12\bZ_{\ge 0}$ for all $i$. Using (\ref{bound-E7}), we immediately see that necessary conditions are
\[\ 2\nu_3<1,\ 2\nu_4<1,\ 2\nu_5<1,\ 2\nu_6<1, \text{ and } \nu_1+\nu_2+\nu_7<1.
\]
From this, we must have 
\[\nu_3=\nu_4=\nu_5=\nu_6=0,\text{ and }\nu_1,\nu_2,\nu_7\in\{0,\frac 12\}\text{ but only one is nonzero}.
\]

This means that there are four remaining points to check:
\begin{itemize}
\item $\nu=0$, this is the origin;
\item $\nu=\frac 12\omega_7$;
\item $\nu=\frac 12\omega_1$;
\item $\nu=\frac 12\omega_2$;
\end{itemize}
But the last two points are both on the hyperplane $\gamma^\vee=1$, where $\gamma^\vee=-\ep_7+\ep_8$ is the highest coroot. On the other hand, the second point, $\nu=\frac 12\omega_7$ satisfies
\[\langle\gamma^\vee,\frac 12\omega_7\rangle=\frac 12<1,
\]
so it lies in the (open) fundamental alcove, which is always unitary.

\subsection{$E_6$} We use $\al_i^\vee$, $1\le i\le 6$ from $E_8$, but the coordinates for the fundamental weights are different:
\begin{equation}
\begin{aligned}
\om_1&=(0,0,0,0,0,-2/3,-2/3,2/3), &\om_2&=(\frac 12,\frac 12,\frac 12,\frac 12,\frac 12,-\frac 12,-\frac 12,\frac 12),\\
 \om_3&=(-\frac 12,\frac 12,\frac 12,\frac 12,\frac 12,-\frac 56,-\frac 56,\frac 56),
 &\om_4&=(0,0,1,1,1,-1,-1,1),\\ \om_5&=(0,0,0,1,1,-\frac 23,-\frac 23,\frac 23), &\om_6&=(0,0,0,0,1,-\frac 13,-\frac 13,\frac 13).
 \end{aligned}
\end{equation}
Write $\nu=\sum_{i=1}^6\nu_i\omega_i\in \CC_0^{\frac 12}$, $\nu_i\in \frac 12\bZ_{\ge 0}$ for all $i$. Such a parameter is hermitian ($w_0(\nu)=-\nu$) if and only if 
\[\nu_1=\nu_6\text{ and }\nu_3=\nu_5.
\]
From \cite[\S3.5]{Ci-E6}, see also \cite[\S7.2.1]{BC-E8},  we see that $r_0=9$ and there is one coroot at level $9$:
\begin{equation}\label{bound:E6}
\beta^\vee=\al_1^\vee+\al_2^\vee+2\al_3^\vee+2\al_4^\vee+2\al_5^\vee+\al_6^\vee=\frac 12(-1,1,-1,1,1,-1,-1,1).
\end{equation}
From this, it is immediate that necessary conditions are:
\[2\nu_3<1,\ 2\nu_4<1,\ 2\nu_5<1,\text{ and }\nu_1+\nu_2+\nu_6<1.
\]
Hence $\nu_3=\nu_4=\nu_5=0$. Since $\nu_1=\nu_6$, it also follows that $\nu_1=\nu_6=0$. We are left with $\nu_2<1$. The only other point besides the origin is $\nu=\frac 12\omega_2$. But this point lies on the hyperplane $\gamma^\vee=1$, where $\gamma^\vee=\frac 12(1,1,1,1,1,-1,-1,1)$ is the highest coroot.

\subsection{Other fundamental representations} We retain the notation from the previous section. Let $V=V(\lambda^\vee)$ be an irreducible finite-dimensional representation of $G^\vee$ with highest weight $\lambda^\vee\in \ft$.

\begin{lemma}\label{l:half-other}
Let  $\nu\in \CC_0^{\frac 12}\subset \ft^\vee$ be given. Then $\nu$ is half-integral from $V(\lambda^\vee)$ if and only if $\langle\lambda,\nu\rangle \in \frac 12\bZ$.
\end{lemma}

\begin{proof}
If $\lambda'^\vee$ is a weight of $V(\lambda)$, then we may write $\lambda'^\vee=\lambda^\vee-\sum_{i=1}^\ell a_i\alpha_i^\vee$, where $a_i\in \bZ_{\ge 0}$. Then 
\[\langle \lambda'^\vee,\nu\rangle=\langle\lambda^\vee,\nu\rangle-\sum_{i=1}^\ell a_i\langle\al_i^\vee,\nu\rangle\in \langle\lambda^\vee,\nu\rangle+\frac 12\bZ,
\]
since  $\nu\in \CC_0^{\frac 12}$ and hence by definition $\langle\al_i^\vee,\nu\rangle\in \frac 12\bZ$. The claim follows.
\end{proof}

\begin{prop}\label{p:half-other} Let $G$ and the rest of the notation be as in Theorem \ref{t:half-adjoint}. 
\begin{enumerate} 
\item[(i)] If $G=PGL(n)$, $n$ even, then $\frac 12\omega_{\frac n2}$ is not half-integral for the defining representation $V(\omega_1^\vee)$;
\item[(ii)] If $G=SO(2n+1)$, then $\frac 12 \omega_1$ is not half-integral for the defining representation $V(\omega_n^\vee)$;
\item[(iii)] If $G=PSp(2n)$, then $\frac 12 \omega_n$ is not half-integral for the spin representation $V(\omega_1^\vee)$;
\item[(iv)] If $G=PSO(2n)$, then $\frac 12 \omega_n$ is not half-integral for the spin representation $V(\omega_1^\vee)$, while $\frac 12\omega_1$ and $\frac 12\omega_2$ are not half-integral for the defining representation $V(\omega_n^\vee)$;
\item[(v)] If $G=E_7$, then $\frac 12\omega_7$ is not half-integral for the $56$-dimensional fundamental representation $V(\omega_7^\vee)$.
\end{enumerate}
\end{prop}

\begin{proof} In light of Lemma \ref{l:half-other}, we only need to check the pairing of the listed points with the corresponding highest weight. In the notation from Theorem \ref{t:half-adjoint} and the subsequent subsections, the relevant pairings are as follows:
\begin{enumerate}
\item[(i)] $\langle \omega_1^\vee, \frac 12\omega_{\frac n2}\rangle =\frac 14$;
\item[(ii)] $\langle \omega_n^\vee, \frac 12\omega_1\rangle =\frac 14$;
\item[(iii)] $\langle \omega_1^\vee, \frac 12\omega_n\rangle =\frac 14$;
\item[(iv)] $\langle \omega_1^\vee, \frac 12\omega_n\rangle =\langle \omega_n^\vee, \frac 12\omega_1\rangle =\langle \omega_n^\vee, \frac 12\omega_2\rangle =\frac 14$;
\item[(v)] $\langle \omega_7^\vee, \frac 12\omega_7\rangle=\frac 34$.
\end{enumerate}

\end{proof}

We can now state our conclusions.

\begin{cor}\label{c:half-integral} Let $s=s_c s_\nu$ be a Satake parameter as in (\ref{e:polar}) and let $L(s)$ be the corresponding spherical $G(k)$-representation. Suppose $L(s)$ is generic.
\begin{enumerate}
\item If $L(s)$ is unitary, then $s$ is half-integral with respect to the adjoint $G^\vee$-representation and the fundamental $G^\vee$-representations listed in Proposition \ref{p:half-other} if and only if $\nu=0$, i.e., if and only if $L(s)=L(s_c)$ is tempered.
\item If there exists a coroot $\alpha^\vee$ such that $\langle \alpha^\vee,s\rangle=q^m$ for some $m\in \bZ\setminus\{0\}$, then $L(s)$ is not unitary.
\end{enumerate}
\end{cor}

\begin{proof}
For (1), clearly $s$ is half-integral with respect to a given $G^\vee$-representation $V$ if and only if $\nu$ is. Without loss of generality, we may assume that $\nu$ is dominant. If $L(s)$ is assumed generic and unitary, we must have $\nu\in \CU_0^g$. By Theorem \ref{t:half-adjoint}, $\nu\in \CU_0^{g,\frac 12}$, a set that was determined explicitly. Using Proposition \ref{p:half-other}, no nonzero point in $ \CU_0^{g,\frac 12}$ can be half-integral for the listed fundamental representations.
The converse is obvious.

Part (2) follows at once from Proposition \ref{p:integral-unitary}.
\end{proof}

\section{Nonsplit unramified groups}\label{s:nonsplit}

In this section, let $G$ be a quasi-simple {\it unramified} group of adjoint type over the nonarchimedean local field $k$. This means \cite[p. 135]{Ca} that
\begin{enumerate}
\item[(a)] $G$ is quasi-split over $k$, and
\item[(b)] there exists an unramified extension $k'$ of $k$ of finite degree $d$ such that $G$ splits over $k'$.
\end{enumerate}
For applications to the automorphic setting, these are the groups that need to be considered. In the previous section, we proved the unitarity results in the case when $G$ is $k$-split. Now, we show how to deduce the analogous results for $G$ unramified from the split case. We follow the approach of \cite[\S6.1-6.2]{BC-ajm}. We note that the assumption that $k$ has characteristic $0$ stated in \cite[\S1.3]{BC-ajm} is unnecessary.

Let $A$ be a maximally split torus of $G(k)$ and $M=C_{G(k)}(A)$, $W(G,A)=N_{G(k)}(A)/M$. Let $K$ be a special maximal compact subgroup attached via the Bruhat-Tits building to $A$. Let $B\supset A$ be a $k$-rational Borel subgroup. Let $X^*(M)$ and $X_*(M)$ denote the lattices of algebraic characters and cocharacters of $M$, respectively. Let $\langle~,~\rangle$ be their
natural pairing. Define the valuation map as before
\begin{equation*}
  \begin{aligned}
\mathsf{val}_M:&~M\to X_*(M),\langle\lambda,\mathsf{val}_M(m)\rangle=\mathsf{val}_k(\lambda(m)),\text{ for all } m\in M,\ \lambda\in X^*(M).    
  \end{aligned}
\end{equation*}
Set $^0\!M=\ker v_M$ and let $\Lambda(M)$ be the image of $\mathsf{val}_M.$ %Similarly, define $^0\!A$, and $\Lambda(A).$ Since $A$ is a torus, we have$\Lambda(A)=X_*(A).$ Moreover, we have $X_*(A)\subset\Lambda(M)\subset X_*(M).$ We have $\Lambda(M)=X_*(M)$ precisely when $M=A$ (and so $G$ is $k$-split).

The group of unramified characters $\widehat {M}^u$ of $M$ (i.e., characters trivial on $^0\!M$) can be identified with the complex algebraic torus $T'^{\vee}=\text{Spec}~\bC[\Lambda(M)]$. 
For every $s\in T'^{\vee}$, let  $\chi_s\in \widehat M^u$ be the corresponding unramified character and let $X(s)=i_{B(k)}^{G(k)}(\chi_s)$ denote the  unramified principal series. 

As before, denote by $L(s)$ the unique irreducible $K$-spherical subquotient. We have
a  one-to-one correspondence $s\leftrightarrow L(s)$ between $W(G,A)$-orbits in $T'^{\vee}$ and irreducible $K$-spherical representations of $G(k)$.

We need to express this bijection in terms of $^L\!G$.  Let
$\Psi(G^\vee)=(X^*(T^\vee),X_*(T^\vee),\Phi^\vee,\Phi)$ be the root datum of $G^\vee$, and let $W$ be its Weyl group.
The inner class of $G$ defines a  homomorphism $\tau:
Gal(\overline k/k)\to \mathrm{Aut}(\Psi(G))$. Given our assumptions of $G$, we know that the image
$\tau(Gal(\overline k/k))\subset  \mathrm{Aut}(\Psi(G))$ is a cyclic group 
generated by an automorphism of order $d$ ($d\in\{2,3\}$) which, by abuse of notation, we also call
$\tau.$ Fix a choice of 
root vectors in $\mathfrak g^\vee$, $X_{\alpha^\vee},$ for $\alpha\in \Phi.$ The automorphism $\tau$ maps the root space
of $\alpha^\vee$ to the root space of $\tau(\alpha^\vee).$ We can normalize $\tau$ so
that 
$\tau(X_{\alpha^\vee})=X_{\tau(\alpha^\vee)}$, for all $\alpha$.
Hence $\tau$ induces an automorphism of $G^\vee$. 

Two elements $x_1,x_2\in G^\vee$ are called {\it $\tau$-conjugate} if there
exists $g\in G$ such that $x_2=g\cdot x_1\cdot \tau(g^{-1}).$ For a
subset $S\subset G$, denote
\[N_G(S\tau)=\{g\in G:~ g\cdot S\cdot \tau(g^{-1})\subset S\}.\]

 The construction of the L-group is such that we have (see
 \cite[section 6]{Bo}):
\begin{equation}
X^*(T'^{\vee})=X^*(T^\vee)^\tau,\quad W(G,A)=W^\tau.
\end{equation}
In particular, we have an inclusion $X^*(T'^{\vee})\hookrightarrow X^*(T^\vee)$,
which gives a surjection $\eta:T^\vee\to T'^{\vee}.$ 
We recall that by \cite[Lemma 6.4, Proposition 6.7]{Bo}, the map 
\begin{equation}
\eta': T^\vee\rtimes\langle\tau\rangle\to T'^{\vee},\quad \nu'((s,\tau))=\nu(s),
\end{equation}
induces a bijection of $(N(T^\vee\tau),\tau)$-conjugacy classes of elements
in $T^\vee$ (hence of semisimple $\tau$-conjugacy classes in $G^\vee$) and $W^\tau$-conjugacy classes of elements in $T'^{\vee}.$

Let
\begin{equation}G^\vee(\tau)=\{g\in G^\vee: \tau(g)=g\}.\end{equation} 

This is a connected reductive group (recall that $G^\vee$ is assumed to be simply connected).  Let  $G(\tau)$ denote the split $k$-adic group whose Langlands dual  is $G^\vee(\tau)$. The explicit cases are listed in Table \ref{t:1}.

\begin{table}[h]\caption{Nonsplit quasisplit unramified $G$ and corresponding split $G(\tau)$.\label{t:1}}
\begin{tabular}{|c|c|c|c|}
\hline
Type of $G$ &Label in \cite{Ti} &Order of $\tau$ &Type of $G(\tau)$\\
\hline
$A_{2n-1}$ &$^2\!A_{2n-1}'$ &$2$ &$B_n$\\
\hline
$A_{2n}$ &$^2\!A_{2n}'$ &$2$ &$C_n$\\
\hline
$D_n$ &$^2\!D_n$ &$2$ &$C_{n-1}$\\
\hline
$D_4$ &$^3\!D_4$ &$3$ &$G_2$\\
\hline
$E_6$ &$^2\!E_6$ &$2$ &$F_4$\\
\hline
\end{tabular}
\end{table}

As in section \ref{classificationsection}, for every $\nu\in \mathfrak t_{\mathbb R}'^\vee=(\mathfrak t_{\mathbb R}^\vee)^\tau$, let $X^G(\nu)=X({\exp(\nu\log(q)})$ be the unramified principal series for $G(k)$. Let also $X^{G(\tau)}(\nu)$ be the unramified principal series for $G(\tau)(k)$. Let $L^G(\nu)$ and $L^{G(\tau)}(\nu)$ be the corresponding spherical subquotients. The following statement is a particular case of \cite[Corollary 6.2.6]{BC-ajm}

\begin{thm}\label{t:quasi-corr} For every $\nu\in \mathfrak t_{\mathbb R}'^\vee$, the irreducible representation $L^G(\nu)$ is (generic) unitary if and only if $L^{G(\tau)}(\nu)$ is (generic) unitary.
\end{thm}

From this we deduce immediately the following analogue of Corollary \ref{c:half-integral}.

\begin{cor}\label{c:unramified}
Suppose $G$ is a quasi-simple unramified nonsplit adjoint $k$-group and that $G$ is not the quasisplit form of $PSU(2n+1)$. Let $s=s_c s_\nu\in T'^\vee$ be a Satake parameter and $L^G(s)$ the corresponding spherical $G(k)$-representation. Suppose $L^G(s)$ is generic and unitary. Then $s$ is half-integral with respect to all finite-dimensional representations of $G^\vee\rtimes\langle \tau\rangle$ if and only if $\nu=0$, i.e., if and only if $L^G(s)=L^G(s_c)$ is tempered.
\end{cor}

\begin{proof}
The generic spherical $G(k)$-representation $L^G(\nu)$ is unitary, hence by Theorem \ref{t:quasi-corr}, $L^{G(\tau)}(\nu)$ is generic unitary. We wish to apply Theorem \ref{t:half-adjoint} and Proposition \ref{p:half-other} to $L^{G(\tau)}(\nu)$. For this, we need to exhibit finite-dimensional representations of $G^\vee\rtimes \langle\tau\rangle$ such that their restrictions to $G^\vee(\tau)$ contain the adjoint representation of $G^\vee(\tau)$ and the representations listed in Proposition \ref{p:half-other}. 

Firstly, the adjoint  adjoint $G^\vee$-representation is $\tau$-stable and contains in its restriction the adjoint representation of $G^\vee(\tau)$.  Next, it remains to discuss the classical groups listed in Proposition \ref{p:half-other}; in all cases $\tau^2=1$. 

\begin{enumerate}
\item $G^\vee=SL(2n,\bC)$, $G^\vee(\tau)=Sp(2n,\bC)$, $G(\tau)=SO(2n+1,k)$. The vector $G^\vee$-representation $V=\bC^{2n}$  restricts to the defining representation $V(\omega_n^\vee)$ of $G^\vee(\tau)$.
\item $G^\vee=Spin(2n,\bC)$, $G^\vee(\tau)=Spin(2n-1,\bC)$, $G(\tau)=PSp(2n-2,k)$. The two spin $G^\vee$-representations restrict both to the spin representation of $G^\vee(\tau)$.  
\end{enumerate} 
In this way, Proposition \ref{p:half-other} can be applied to $G^\vee(\tau)$ and it follows that $\nu=0$.

\end{proof}

\begin{remark}\label{PSU}
The methods in this section are insufficient to establish the claim in Corollary \ref{c:unramified} in the case when $G$ is the quasisplit form of $PSU(2n+1)$. In that case, $G^\vee=SL(2n+1,\bC)$, $G^\vee(\tau)=SO(2n+1,\bC)$, and $G(\tau)=Sp(2n,k)$. The adjoint representation of $G^\vee$ contains in its restriction the adjoint representation of $G^\vee(\tau)$, which means that the argument in the proof of the corollary implies that the only possible half-integral Satake parameters for $G(k)$ have $\nu=0$ or $\nu=\frac 12\omega_n$, where $\omega_n$ is the $G(\tau)$-weight (the notation as in subsection \ref{sub:psp}). But in the case of $PSp(2n,k)$, the point $\nu=\frac 12\omega_n$ is ruled out by the spin representation of the dual group $Spin(2n+1,\bC)$, see Proposition \ref{p:half-other}(iii), and the spin representation does not descend to $SO(2n+1,\bC)$ (in particular, it can't appear in the restriction of a finite-dimensional $SL(2n+1,\bC)$-representation). 
\end{remark}

\end{document}